\documentclass[12pt]{amsart}
\usepackage{amsmath}
\usepackage{amssymb}
\usepackage{amsthm}
\usepackage{color,comment}
\usepackage[all]{xy}
\usepackage{graphicx}
\usepackage{amscd,color}

\newtheorem{thm}{Theorem}[section]
\newtheorem*{starthm}{Theorem}

\newtheorem{prop}[thm]{Proposition}
\newtheorem{lemma}[thm]{Lemma}
\newtheorem{defn}{Definition}

\newtheorem{remark}{Remark}[section]

\def\F{\mathcal F}

\def\AA{\mathbb A}
\def\BB{\mathbb B}
\def\GG{\mathbb G}

\def\LL{\mathbb{L}}
\def\LLA{\mathbb{L}_{\mathbb A}}
\def\LLB{\mathbb{L}_{\mathbb B}}
\def\EE{\mathbb{E}}
\def\WW{\mathbb{W}}
\def\YY{\mathbb{Y}}

\def\F{\mathcal F}

\def\H{\mathcal H}

\def\CC{\mathbb C}

\def\HH{\mathbb H}

\def\QQ{\mathbb Q}

\def\XX{\mathbb X}

\def\LL{\mathbb L}

\def\iota{\mathcal{INV}}
\def\Isom{\rm Isom}
\def\INV{\mathcal I_{hyp} }

\title{Winding and Unwinding and  essential intersections in $\HH^3$}
\author{Jane Gilman}
\address{Department of  Mathematics and Computer Science, Rutgers University, Newark, NJ
 07079}
 \email{gilman@rutgers.edu}
 \thanks{Some of this work was carried out while the author was a visitor at ICERM}

\date{\today}

 \author{Linda Keen}
\address{Department of  Mathematics,  The Graduate
Center and Lehman College,  CUNY, New York, NY 10016}
 \email{LKeen@gc.cuny.edu}
\thanks{This work was supported by a CUNY Collaborative grant and grants from PSC-CUNY}
\begin{document}
\begin{abstract}

Let $\GG = \langle \AA,\BB\rangle$ be a non-elementary two generator subgroup of the isometry group of $\HH^2$. If $G$ is discrete, free and generated by two hyperbolic elements with disjoint axes, its quotient is a pair of pants and in \cite{GKwords} we produced a recursive formula for the  number of essential self intersections (ESIs) of any primitive geodesic on the quotient.  An alternative non-recursive proof of the formula was given  in \cite{Vidur}.   We used the formula to give a geometric description of how the primitive geodesic winds around the ``cuffs" of the pants.

On  a pair of pants, an ESI is a point  of a ``seam" where a primitive geodesic has a self-intersection.   Here we generalize the definition of  ESI's  to non-elementary purely loxodromic two generator groups $G \subset \Isom(\HH^3)$ which are discrete and free.  We also associate an open geodesic to each generalized ESI we call a {\em Connector.}  In the quotient manifold the connectors correspond to ``opened up self-intersections".  We then define a class of groups we call {\em Winding groups};  the connectors for these groups give us a way generalize   the concept of winding of primitive geodesics in the quotient manifold to this context.  For winding groups,
 the ESI's and connectors   satisfy the  same formula as they do in  the two-dimensional case.  Our techniques involve using the {\sl stopping generators} we defined in \cite{GKwords}  for the group $\GG$ to make it a model.  We then use the  model   to obtain an algebraic  rather than geometric definition for ESI's and generalize from the algebraic definition.
\end{abstract}
\maketitle

\section{Introduction}

In \cite{GKwords} we considered some geometric implications of the Gilman-Maskit discreteness algorithm, \cite{GM}, for non-elementary two generator subgroups of $\Isom(\HH^2)$ whose generators were hyperbolic with disjoint axes and  whose product was also hyperbolic.   We showed that when the algorithm determines that the group is discrete and free, the quotient is a three holed sphere and the algorithm  has determined discreteness by finding a special pair of {\sl stopping generators} that correspond to the only three simple geodesics on the quotient surface, those around the waist and the cuffs.   Every other geodesic on the quotient  winds around the waist and cuffs and  has self-intersections.  We concentrate our attention on the {\em primitive geodesics} on the quotient. A primitive element of the fundamental group of the surface is one that together with one other element generates the fundamental group. A {\sl primitive curve}  is the geodesic in the free homotopy class of such a primitive element. A pair of such geodesics coming from a pair of elements that generate the fundamental group is a {\sl  primtive pair}.   These conjugacy classes of primitive geodesics can be identified with the set of rational numbers, \cite{Enumeration,  KS1}.

In discreteness considerations, it is standard to consider a  group generated by three involutions in which the two  generator group sits as a subgroup of index two. This can be done for two generator groups both in $\Isom(\HH^2)$ and $\Isom(\HH^3)$.   For the subgroups of $\HH^2$   under consideration here,  the fixed points of these involutions project to geodesics on the surface we call {\sl seams}.   Some of the self-intersections of the primitive geodesics necessarily lie on these   seams.  In \cite{GKwords}  we called these {\em Essential Self-Intersections} to differentiate   them from  other self-intersections.    We produced a recursive formula to count the essential self   intersections of the primitive geodesics.    We showed that in the quotient, one could view the primitive geodesics as curves that {\sl wind around}  the simple geodesics and that the winding reflects the rational identified with the geodesic.     A non-recursive version of the formula for essential self-intersections was given in \cite{Vidur} along with a proof of the recursive version.

In this paper we  generalize the concept of Essential Self-Intersections (ESI's) of   primitive geodesics to non-elementary  two generator free discrete subgroups  of $\Isom(\HH^3)$.  The hyperbolic quotient manifold is now a genus two handlebody and the group elements correspond to non-trivial free homotopy classes of curves in the manifold.   Each  class has a unique geodesic and to simplify the exposition, we do not distinguish between the  group element and the geodesic in the free homotopy class it represents. Also, for the sake of brevity and clarity  of the exposition we restrict our discussion to purely loxodromic groups.  The results can be extended to groups containing parabolic elements as our two dimensional treatment did.

 We note that in \cite{Vidur} the original   definition of ESI's and the recursive formulas were  extended to discrete groups that included elliptic elements  and it is likely that these results can also be extended to the three dimensional case.

  Generically, the projections of the axes of primitive elements of the group are simple closed curves in the quotient manifold so the there are no self-intersections.   Therefore, in order to generalize the ESI's we need to understand what happens to the intersection points when the subgroup of $\Isom(\HH^2)$ is deformed into a subgroup of $\Isom(\HH^3)$; that is, to understand how the geodesic ``opens up'' under the deformation.   To this end we associate open geodesic arcs we call  {\em connectors} to
  the generalized ESI's  we define for these more general groups. The connectors are perpendicular to a primitive geodesic in the quotient manifold and represent the opened up intersection point.

 In \cite{GKhyp} we showed that the involutions of the degree two extension of the group induce a {\em hyperelliptic}  involution of the quotient manifold that  generalizes the hyperelliptic involution of its boundary.  The fixed set of this involution consists of three bi-infinite open geodesic arcs and  is a  generalization of  the seams.  We prove

\begin{starthm}[Theorem~A] The connectors associated to the generalized ESI's are arcs joining points on the seams  and they are invariant under the generalized hyperelliptic involution;  they satisfy the same two formulas as ESI's satisfy in the two dimensional case.
\end{starthm}

We then define a certain class   of these non-elementary rank two discrete free discrete subgroups of $\Isom(\HH^3)$ that  we call {\em Winding groups}.
For these groups  we can define an analogue of the stopping generators and a geometric analogue of winding.  For this class of groups we prove

\begin{starthm}[Theorem~B]  For every primitive geodesic in the quotient manifold, there is an analogue of the {\sl winding} we found in the two dimensional case. This winding reflects the rational number identified with the geodesic. \end{starthm}

  The paper is organized as follows.   In Section~2 we give some background, define the terms we use and the enumeration scheme. In Section~3 we discuss the model group $\GG=\langle \AA,\BB \rangle \subset \HH^2$ and  review  Essential Self Intersections for this fixed group as an example of a non-elementary discrete free rank two purely hyperbolic group in $\Isom(\HH^2)$ whose generators have disjoint axes.  In Section~4   we use  the natural isomorphism from the  model group to the discrete free purely loxodromic subgroup $G$ of $\Isom(\HH^3)$ to label elements of $G$ and its degree two extension $\tilde{G}$.  We then use the labelling  to define the  generalized ESI's and the connectors and   we prove Theorem~A for non-elementary discrete,  free, purely loxodromic groups.  In Section~5 we define a class of groups for which we prove Theorem~B.  We call these  {\em winding groups};  they have
   the properties that the quotient manifold has a rational pleating locus (defined in that section) and any pair of pleating curves are primitive geodesics that can be represented by a pair of generators of the group.   We expect that there are other classes of groups for which Theorem~B holds but we leave that discussion for the future.

 \section{Background and Terminology}\label{background}
We use the following terminology.  We consider $\HH^2 \subset \HH^3$ so all of the following makes sense for both the two and three dimensional settings.
   If    $M$ is a geodesic in $\HH^3$ we let $H_M$ denote the half-turn about $M$; this isometry is an orientation preserving  involution whose set of fixed points is the geodesic $M$.  If $X$ is an isometry we let $Ax_X$ denote the invariant geodesic joining its fixed points; these may be on the boundary of $\HH^3$.   This geodesic is called its axis.  If $X$ has only one fixed point it is parabolic and the axis is the fixed point.  When we restrict to $\HH^2$, the axis of an elliptic is also a point.   For brevity and clarity of exposition,  in this paper, we will always assume that the isometries we encounter are not parabolic and those that are elliptic have order two   and are  half-turns.  An element of order two is a half-turn if and only if its axis is fixed pointwise.

   If $M$ is any  geodesic orthogonal to $Ax_X$, there is another geodesic $M'$ orthogonal to $Ax_X$ such that $X$ is the product of the two half-turns,  $X=H_{L_M}H_{L_{M'}}$.    Note that any two geodesics in $\HH^3$ have a unique common perpendicular. Any element of $\Isom(\HH^3)$ can be factored in an infinite number of ways as a  the product of   half-turns about a pair of geodesics perpendicular to its axis;   if $X$ is loxodromic, the hyperbolic distance between the pair along the axis is half the translation length while if $X$ is elliptic they intersect the axis at the same point  and if $X$ is parabolic they intersect at a point on the boundary of the hyperbolic space.

  We are interested in discrete free groups $G$ of rank two;  we always assume the groups are non-elementary and, for the sake of brevity, give no further mention of that assumption.
  We will be concerned with {\em primitive elements of $G$}.  For a free group of rank $n$, an element is {\em primitive}  if it is part of a minimal generating set \cite{MKS}.    We call a geodesic {\em primitive}  if it is in the free homotopy class of the fundamental group of the quotient represented by a  primitive element.   In this paper,  a minimal generating set is a {\em primitive pair}.

  If $G$ is generated by $A$ and $B$, the axes of $A$ and $B$ have a common perpendicular $L$  and there are a unique pair of geodesics $L_A$ and $L_B$ such that $A= H_L \circ H_{L_A}$ and $B = H_L \circ H_{L_B}$.

  We denote by $\tilde{G}$  the group generated by the three half-turns
  \[ \tilde{G}= \langle H_L, H_{L_A}, H_{L_B} \rangle. \]
  It is clear that $G \subset \tilde{G}$ and that  the index of $G$ in $\tilde{G}$ is $2$.  It is a classical result that a group and a normal subgroup of finite index  are simultaneously discrete and or non-discrete.

  We have $A^{-1}B = H_{L_A} \circ H_{L_B}$ and its axis is the common perpendicular to $L_A$ and $L_B$;  also $AB= H_L \circ H_{L_A} \circ
  H_L \circ H_{L_B}$  and its axis is the common perpendicular to $ H_L \circ H_{L_A} \circ H_L$ and $H_{L_B}$.
  The cyclically ordered set of  lines
  \[ Ax_A, L_A,  Ax_{A^{-1}B},  L_B, Ax_B, L \]
  consists of six geodesics with the property that each is orthogonal to both of its neighbors.   They form a {\em right-angled hexagon} as defined by Fenchel in \cite{Fenchel}.    We denote this hexagon by $\H(A,B)$;  it has three {\em axis sides} and three {\em half-turn sides}.

   Note that in our discussion of subgroups of $\Isom(\HH^2)$ below, we always assume the axes of the given generators of our group are disjoint.   The  case where the axes intersect is subsumed in   the general discussion of subgroups of $\Isom(\HH^3)$ above.  Again for the sake of brevity, we leave a more detailed discussion of results about essential intersections for  groups with intersecting axes for another time.

If $G$ acts invariantly on $\HH^2$ and is discrete and free, and if the axes of the generators are disjoint, the quotient $S=\HH^2/G$ is a three-holed sphere or pair of pants.   A primitive  curve on $S$ is the image of the axis of a primitive element in the group.  The quotient $S$ is a two-fold covering of the quotient $\tilde{S}=\HH^2/\tilde{G}$ which is an orbifold;  we can visualize it as identifying the front and back sides of a pair of pants.  The orbifold points are the projections of the half-turn lines; we call these {\sl the seams} of the pants.  The involution that identifies the front and back is induced by any one of the half-turns in $\tilde{G}$.

\subsection{Enumeration of Primitive Elements}

As we indicated in the introduction there are various ways to enumerate the conjugacy classes of  primitive words.  Each is an inductive scheme that associates a representative of the conjugacy class of primitive words to a  rational number $p/q$.  The scheme lists a word but not its inverse. The generators $A,B$ are assigned to $0/1$ and $1/0$ respectively and the representative word is built up from a sequence of steps using the Farey tessellation of the upper half plane.  In \cite{Enumeration} we denoted the representative words by $E_{p/q}$  and called the enumeration scheme produced there  the {\em palindromic enumeration scheme}. In that scheme, if $p+q$ is odd ($pq$ is even),  then the word $E_{p/q}$ is a palindrome in $A$ and $B$ while if $p+q$ ($pq$ odd)  is even, the word $E_{p/q}$ is a product of the two words in the previous two steps of the scheme; these last two words are palindromes and the product is chosen so that the next word in the scheme is again a palindrome.   This scheme reflects the $H_L$ symmetry of the group because the axes of the palindromic words in this scheme are all perpendicular to the half-turn line $L$.

 We will use the palindromic scheme here and assume $p/q>0$.      Suppose first that $E_{p/q}$ is a palindrome.  Then    $p+q=2n+1$ and $E_{p/q}$ has the form $WU \overline{W}$ where $U \in \{A,B\}$ and $W=X_1X_2 \ldots X_n$ is a word of length $n=(p+q-1)/2$, $X_j \in \{A,B\}$ and $ \overline{W}=X_n \ldots X_1$.   If $E_{p/q}$ is not a palindrome so that
  $p+q=2n$ is even, we write $E_{p/q}=X_1 \ldots X_{2n}$.   There are a unique pair of rationals $l/m,r/s$ with $| ls-mr |=1$ such that   $\frac{p}{q}=\frac{l}{m}+\frac{r}{s}$.  We set   $E_{p/q}=E_{l/m}E_{r/s}$ or  $E_{p/q}=E_{r/s}E_{l/m}$ where the choice is made so that the next words in the scheme,  $E_{\frac{p+l}{q+m}}$ and $E_{\frac{p+r}{q+s}}$ are palindromes.   We omit the details here.  The interested reader  can find them in \cite{GKwords, Enumeration}.

  The forms that the words take can be made more explicit, but we will not need them here.  The only property  of these words  that we will need  is
 \begin{remark} \label{samesign}     In the identification of a rational  and a primitive word,   each time a generator appears in a given word, its exponent has the same sign. That is, in the palindromic enumerations  scheme,  if $p/q>0$ all of the exponents of $A$ and $B$ are positive whereas if $p/q<0$ the exponents of $A$ are all negative and the exponents of $B$ are all positive.  Since $B^{-1}$ never appears, no word is an inverse of another.
 \end{remark}

\section{$\HH^2$: The Model Group}
In this section we assume that we are given a  fixed group $\GG=\langle \AA,\BB \rangle  \subset \Isom(\HH^2)$  where the axes of $\AA$ and $\AA\BB$ are disjoint and that $\AA,\BB$ are stopping generators for the Gilman-Maskit algorithm.   By results of \cite{GKwords}, this means that their axes and the axis of $\AA^{-1}\BB$ project to simple curves on the quotient surface $S$ which is a three-holed sphere.   All other axes of elements of $\GG$ project to curves with self-intersections. We also consider the degree two extension  $\tilde\GG=\langle H_{\LL}, H_{\LLA}, H_{\LLB} \rangle$ of $\GG$.

The condition that  $\AA,\BB$ are stopping generators also means that the hexagon $\H(\AA,\BB)$ is convex in $\HH^2$ and moreover, that
\[ \F= \H(\AA,\BB) \cup H_{\LL}(\H(\AA,\BB)) \]
is a fundamental domain for $\GG$.   It has the property that it is symmetric about the line $\LL$ where the symmetry is the half-turn $H_{\LL}$.
Note that we use an overline to denote the image under $H_{\LL}$ so that $\overline{\H(\AA,\BB)}=H_{\LL}(\H(\AA,\BB))$ is also a convex hexagon with three axis sides,  the axes of $\AA$, $\BB$ and $\BB\AA^{-1}$, and three half-turn sides $\LL$, $\LL_{\bar{\AA}}=H_{\LL}(\LLA)$ and $\LL_{\bar{\BB}}=H_{\LL}(\LLB)$.   The hexagons $\H(\AA,\BB)$ and $\overline{\H(\AA,\BB)}$ are each fundamental domains for $\tilde{G}$.

In this model group the axes of the primitive elements $\EE{p/q}$ have geometric properties we want to look at.  We assume throughout that $p/q>0$.  To obtain  the primitive words in the group in which $\AA$ appears with negative exponent and $-1<p/q<0$, we need to replace $\F$ by the fundamental domain for
$\GG$ given by
\[ \F_{\BB}= \H(\AA,\BB) \cup H_{\LL_{\BB}}(\H(\AA,\BB) \]
and in the rest of the discussion,   to replace $\LL$ by $\LLB$.   For $p/q<-1$, we interchange the roles of $\AA$ and $\BB$; that is we   replace $\F$ by the fundamental domain for
$\GG$ given by
\[ \F_{\AA}= \H(\AA,\BB) \cup H_{\LL_{\AA}}(\H(\AA,\BB)\]
and in the rest of the discussion replace $\LL$ by $\LLA$. We will not carry this out but leave it to the reader.

 Some of the propositions below and their proofs appear in one or more of \cite{Enumeration, KS1, Vidur}, but because of their importance to the new definitions in $\HH^3$, we  state the results using the current notation and give detailed proofs.

 \begin{prop}\label{thru F} The axis of $\EE_{p/q}$ intersects  $\partial\F$ in half-turn sides.  \end{prop}
 \begin{proof}  One of the intervals on the circle containing the limit set of $\GG$ bounded by the fixed points of $\AA$ is an interval of discontinuity.  Similarly, one of the intervals bounded by the fixed points of $\BB$ is, as is one of the intervals bounded by the axes of $\AA^{-1}\BB$ and $\BB\AA^{-1}$.  Since these axes  are mutually disjoint and the intervals of discontinuity are mutually disjoint, they are separated by the axes.   The half-turn lines on $\partial\F$  are perpendicular to these axes so each has endpoints inside  the intervals of discontinuity.  No axis of a group element can enter the half planes bounded by these intervals of discontinuity so it cannot intersect any of the axis sides of $\F$.
  \end{proof}

  Note that the same argument shows that the axis any element of $\GG$ cannot intersect any copy of $\F$ in an axis side.

  \begin{prop}\label{intersectL} The axis of $\EE_{p/q}$ intersects $\LL$ inside $\F$. \end{prop}
  \begin{proof}  By Proposition~\ref{thru F} the axis of $\EE_{p/q}$ enters and leaves $\F$ through half-turn sides. Write $\EE_{p/q}$ as a product of half-turns.  If, for example, it enters through $\LLA$ and leaves through the sides $\LLB$ the half-turn expression for
  $\EE_{p/q}$ would contain an adjacent pair  $H_{\LLA}H_{\LLB}$ and hence the expression in terms of $\AA$ and $\BB$ would contain $\AA^{-1}\BB$.   This cannot happen if $p/q>0$ by Remark~\ref{samesign}.    The same contradiction would appear in every instance where the axis did not intersect $\LL$.
  \end{proof}

    Our enumeration scheme reflects the symmetry of the domain $\F$ about $\LL$.  We want to exploit this to label the set of copies of $\F$ that a fundamental segment of the axis of $\EE_{p/q}$ passes through.  Suppose first that $p+q=2n+1$ and $\EE_{p/q}$ is a palindrome so that its axis is symmetric about $\LL$.   This means that if $Ax_{\EE_{p/q}}$ leaves $\H(\AA,\BB)$ through $\LLA$, it leaves $\overline{\H(\AA,\BB)}$ through $\LL_{\bar{\AA}}$ and similarly for $\LLB$.
 In addition to passing through $\F$, a symmetric fundamental segment of the axis passes through $n$ copies of $\F=\F_0$ on one side of $\LL$ and $n$ copies on the other.

 We label the sequence of $n$ copies of $\F=\F_0$,  starting with the one adjacent to  $\H(\AA,\BB)$ by
 $\F_{j}, j=1, \ldots, n$ and the sequence of  $n$ copies of $\F=\F_0$,  starting with the one adjacent to  $\overline{\H(\AA,\BB)}$ by
 $\F_{-j}, j=1, \ldots, n$,.  The axis of $\EE_{p/q}$ leaves $\F_0$ through the half-turn sides $L^{p/q}_{1}$ and $L^{p/q}_{-1}$ that as it enters $\F_1$ and $F_{-1}$ respectively.  It
 leaves each $\F_{j}$, $j=-n, \ldots, n$  through a half-turn side that is the image of one of the four half-turn sides of $\F$ by an element of $\GG$.   It also intersects a conjugate of $\LL$ in each copy of $\F$.   We could use the technique of cutting sequences to use groups elements to label these side but the notation is clumsy and we will not need it here  (see \cite{KS1, Vidur}.)  Instead,  we label  the copy of $\LL$ in $\F_j$ by $\LL^{p/q}_{2j}$ and the boundary side through which it leaves $F_j$ by $L^{p/q}_{2j+1}$.     We   label the intersection point of $Ax_{\EE_{p/q}}$ with each $\LL^{p/q}_j$ by $Q^{p/q}_j$, $j=-2n-1, \ldots, 2n+1$.

 Remember that  for each $j$, $j=-2n-1, \ldots, 2n+1$ there is a unique word $\WW^{p/q}_j$ such that
  $\LL^{p/q}_{j}=\WW^{p/q}_j(\YY)$, $\YY \in \{\LL, \LLA, \LLB \}$ so we could use these words as our labels.   The only information  a labelling by group elements instead of indices would give us is that it would be obvious that $\EE_{p/q}(\LL^{p/q}_{-2n-1})=\LL^{p/q}_{2n+1}$ and $\EE_{p/q}(Q^{p/q}_{-2n-1})=Q^{p/q}_{2n+1}$.    We take this for granted here.    It then follows that
 $\EE_{p/q} = H_{\LL^{p/q}_{-2n-1}}H_{\LL}$ so that the axis is perpendicular to $\LL^{p/q}_{-2n-1}$ and $\LL^{p/q}_{2n+1}$ at $Q^{p/q}_{-2n-1}$ and at $Q^{p/q}_{2n+1}$ respectively.

  We also  have
 \begin{prop}\label{nonright}
 The angles  at  intersection   points  $Q_j^{p/q}$, $0<j<2n+1$ are not right.
 \end{prop}
 \begin{proof}  Suppose that for some $0< |j|  \leq 2n$ the angle between $L^{p/q}_{j}$ and $Ax_{\EE_{p/q}}$ were right.  Then the product $H_{L_j}H_{\LL}$ would be an element sharing fixed points with $\EE_{p/q}$.  Since the group is discrete,  this  would imply that there is an integer $m>1$ such that    $(H_{L_j}H_{\LL})^m=\EE_{p/q}$.  This cannot be since $\EE_{p/q}$ is primitive.
 \end{proof}

In \cite{Enumeration} we noted that if $P_1$ and $P_2$ were palindromes in the scheme, then the common perpendicular to $P_1P_2$ and $P_2P_1$ would intersect the common orthogonal $\LL$ to $P_1$ and $P_2$  and thus determine a point on it. We need an even more detailed analysis here.

That is we now suppose  $p+q=2n$ is even so that $\EE_{p/q}=\EE_{l/m}\EE_{r/s}$ where $\frac{p}{q}=\frac{l}{m}+\frac{r}{s}$ with $| ls-mr |=1$.  In this case a fundamental segment  of $Ax_{\EE_{p/q}}$ goes through an even number of copies of $\F$.

We showed in \cite{Enumeration} that  $\EE_{p/q}$ is a  palindrome in $A$ and $AB$ or $B$ and $AB$ so that its axis is perpendicular to conjugates of both of the half-turn sides $\LLA$ and $\LLB$  of $\F$  but it is not perpendicular to $\LL$.
 Now we also consider
 \[ (\tilde{\EE}_{p/q})^{-1}=H_{\LL}\EE_{p/q}H_{\LL} = H_{\LL} \EE_{l/m}\EE_{r/s}  H_{\LL}=( \EE_{l/m})^{-1}(\EE_{r/s}){^-1}=(\EE_{r/s}\EE_{l/m})^{-1} \]
which gives us the other product,  $\tilde{\EE}_{p/q}=\EE_{r/s}\EE_{l/m}$.   Note that this  is   a conjugate of $\EE_{p/q}$,
\[  \EE_{l/m}\EE_{r/s}= \EE_{r/s}^{-1}\EE_{r/s} \EE_{l/m}\EE_{r/s}. \]
In this case we  follow the axis of  $\EE_{l/m}\EE_{r/s}$ as it leaves $\H(\AA,\BB)$ and mark off the $n-1$ adjacent fundamental domains it passes through.   If it is not perpendicular to either of the sides through which it leaves $\F$,   we  go in both directions until we reach copies of $\F$ where the axis is perpendicular to the half-turn side it lands on.  Starting with the furthest copy on either side,  we label these $\F_{j}, j=1, \ldots, n$ and label conjugates of the half-turn side $\LL$ in each copy of $\F$ by $L^{p/q}_{2j}$ and the half-turn sides through which it leaves $\F_j$ by $L^{p/q}_{2j+1}$.   Note that $\F=\F_{j}$ for one of the $j$'s, say $j^*$.

Next we follow the axis of  $\EE_{r/s} \EE_{l/m}$ as it leaves $\overline{\H(\AA,\BB)}$ and mark off the  $n-1$ adjacent fundamental domains it passes through until we reach copies of $\F$  where the axis is perpendicular to the half-turn side it lands on and, starting with the furthest copy either side, label them  by $\F_{-j},  j=1, \ldots, n$;  we label conjugates of the half-turn side $\LL$ in each copy of $\F$ by $L^{p/q}_{-j}$ and the half-turn sides through which it leaves $\F_{-j}$ by $L^{p/q}_{-2j-1}$.  As above we label the intersection points of the axes and the half-turn lines respectively by $Q^{p/q}_j$, $j=-2n-1, \ldots,-1, 1, \ldots  2n+1$.   Note that by the symmetry under $H_{\LL}$ we have $\F_{j*}=\F_{-j*}$

Now each axis passes through $n-1$ disjoint fundamental domains and each has a segment in $\F=\F_{j*}=\F_{-j^*}$.    Again, if we were to use group elements to label the sides we would see that the union of the segments along the axis of $\EE_{p/q}$ and the conjugate of the union of the segments along $\tilde{\EE}_{p/q}$ by $\EE_{r/s}$  give us a fundamental segment for $\EE_{p/q}$.

Note that Proposition~\ref{nonright} holds for both these axes as well.
We also have

  \begin{prop} For $p+q$ even, the axes of  $\EE_{p/q}$ and $\tilde{\EE}_{p/q}$  intersect $\LL$ inside $\F$ in the same point.
 \end{prop}
 \begin{proof} A half-turn about a geodesic in the plane $\HH^2$ is a reflection in the geodesic.  It follows that  since $Ax_{\EE_{p/q}}$ intersects $\LL$, its reflection also does and since the points of $\LL$ are fixed, the intersection with $\LL$ is a fixed point of the reflection.  Because these axes are not orthogonal to $\LL$ the intersection is not orthogonal.
 \end{proof}

This proposition implies that the two points we counted in $\F$ where the axes pass through $\LL$ are actually the same point.
\subsection{  Definition of ESI's.  }

We have shown that corresponding to any primitive $\EE_{p/q}$  of   the palindromic enumeration there is a unique set of $4n+3$ half-turn geodesics.
An ESI is a pair $( L_j^{p/q},Ax_{\EE_{p/q}})$. This  intersection of the axis of $\EE_{p/q}$ and $L^{p/q}_j$ determines a point in $\HH^2$. If  $\XX$ is any primitive element of the group, then  $\XX$ is conjugate to $\EE_{p/q}$ for some $p/q$. Let $\XX = \WW \EE_{p/q} \WW^{-1}$ for  $\WW \in \GG$.  Then in the quotient, the image of the point determined by the pair  $(\WW(L_j^{p/q}), Ax_{\XX})$ is the same as the image of the point determined by the pair $( L_j^{p/q},Ax_{\EE_{p/q}})$.

We say $P \in \HH^2$ is equivalent to $Q^{p/q}_j$ for some $p/q$ and $j$ if there is $\WW \in \GG$ such that $P=\WW(Q_j^{p/q})$.
We denote the equivalence class by square brackets:
\[  [Q_j^{p/q}]= \{\WW(Q_j^{p/q}) \}_{\WW \in \GG}.  \]

  \begin{defn}\label{esi1}    Let
  \[[ESI_{p/q}] = {\bigsqcup}_{j=-2n+1}^{2n-1}  [Q_j^{p/q}] \]
  where $n=[p+q-1]/2$ and $ {\bigsqcup}_{j=2n-1}^{2n-1}$ means that if  $(p+q)$ is odd we omit the term $j=0$ and if
  $p+q$ is even we omit the term $j=j*$.
   Thus we have removed those $j$ such the angle between $Ax_{\EE_{p/q}}$ and $L^{p/q}_j$ at $Q^{p/q}_j$ is right.  These
 are the {\em Essential Intersection points or ESI set} of $\EE_{p/q}$.  The full {\em ESI set of $\GG$} is
 \[ ESI(\GG) = \bigcup_{p/q}  [ESI_{p/q}].   \]\end{defn}

 The following theorem was proved in \cite{GKwords} using recursion and again in \cite{Vidur} where the proof was non-recursive.

 \begin{thm}\label{thm1}  For any $p/q >0$, the number of ESI points of $\EE_{p/q}$ is $2(p+q-1)$.  The number of points counted on the quotient surfaces is $p+q-1$.  \end{thm}

The points $Q^{p/q}_j$ are uniquely defined by the pair of geodesics $(L^{p/q}_j,Ax_{\EE_{p/q}})$.   Thus if $Q^{p/q}_j$ is equivalent to $\WW(Q_j^{p/q})$, for some $\WW \in \GG$, then $\WW(Q_j^{p/q})$ is uniquely defined by the pair of geodesics $(\WW(L^{p/q}_j),\WW(Ax_{\EE_{p/q}}))$.

 This gives us an equivalence relation and  we again denote the equivalence class by square brackets  $[( L^{p/q}_j,Ax_{ \EE_{p/q}})]$.

Another characterization of the ESI set of $\GG$ is then given by
 \begin{defn}\label{esi2}    The set of pairs, ${\bigsqcup}_{j=-2n+1}^{2n-1}[( L^{p/q}_j,Ax_{ \EE_{p/q}})]$ , where $n=[p+q-1]/2$  and $\bigsqcup$ has the same meaning as in Definition~\ref{esi1}
 are the {\em Essential Intersection points or ESI set} of $\EE_{p/q}$.  The full set is
   \[ ESI(\GG) = \bigcup_{p/q} {\bigsqcup}_{j=-2n+1}^{2n-1}[( L^{p/q}_j,Ax_{ \EE_{p/q}})] \]
  \end{defn}

The equivalence of these definitions is clear since the points $Q^{p/q}_j$ are just the intersection points of the pairs of geodesics.   The advantage of the second definition, as we will see, is that we can generalize it to other groups besides the model group.

\subsection{Loops and Winding in $\HH^2$}
Under the projection $\pi:\HH^2 \rightarrow \HH^2/\GG=S$,   the axes $Ax_{\AA}$, $Ax_{\BB}$ and $Ax_{\AA^{-1}\BB}$ project to three simple geodesics $\gamma_{0}=\pi(Ax_{\AA}), \gamma_{\infty}=\pi(Ax_\BB)$ and $\gamma_{-1}=\pi(Ax_{\AA^{-1}\BB})$ which we now label by their corresponding rational.   These are the only simple geodesics on $S$.   The half-turn lines $\LL$, $\LLA$ and $\LLB$ project to bi-infinite geodesics $\pi(L)=\ell, \pi(\LLA)=\ell_{\AA}$ and $\pi(\LLB)=\ell_{\BB}$ such that $\ell$ intersects $\gamma_0$ and $\gamma_{\infty}$ orthogonally and is disjoint from $\gamma_{-1}$,  $\ell_{\AA}$ intersects $\gamma_0$ and $\gamma_{-1}$ orthogonally and is disjoint from $\gamma_{\infty}$ and $\ell_{\BB}$ intersects $\gamma_{\infty}$ and $\gamma_{-1}$ orthogonally and is disjoint from $\gamma_{0}$.   The half-turn $H_{\LL}$ on $\HH^2$ induces an involution $\INV$ of $S$.  The geodesics $\{\ell,\ell_\AA,\ell_\BB\}$ are the seams of $S$.

 Thus we can write $\gamma_0=\sigma_0 \cup \INV(\sigma_0)$  where $\sigma_0$ is the segment   from $\ell$ to $\ell_{\AA}$ to see  that it is made up of  two segments that form a loop.  Similarly  $\gamma_{\infty}= \sigma_{\infty} \cup \INV(\sigma_{\infty})$  is another loop.   Each of these loops intersects exactly two of the seams once.

Write  $\gamma_{p/q}=\pi(Ax_{E_{p/q}})$.

\begin{lemma}\label{loops} For every $p/q$, $\gamma_{p/q}$ can be written as a union of $p+q$ loops  each intersecting $\ell$ and either $\ell_\AA$ or $\ell_\BB$.
\end{lemma}
\begin{proof}     Consider the segment  of $Ax_{\EE_{p/q}}$ from $Q_j^{p/q}$ to $Q^{p/q}_{j+1}$.   One of its endpoints lies inside a fundamental domain $\F_j$ and the other on its boundary.   Denote   the projection of the segment   by $\sigma^{p/q}_j$;   it is simple and  one of its endpoints  is on $\ell$ and the other is on either $\ell_\AA$ or $\ell_\BB$.  Suppose for argument's sake that the other end is $\ell_\AA$.   The union $\sigma_j(p/q) \cup \INV(\sigma_j^{p/q})$ is a closed loop on $S$.  Since it contains one segment going from $\ell$ to $\ell_A$ and another from $\ell_A$ to $\ell$, it is homotopic to $\gamma_0$ on $S$.

Since we can do this for every adjacent pair of points, we obtain $p+q$ loops, each homotopic to either $\gamma_0$ or $\gamma_{\infty}$.
\end{proof}

 \begin{remark}  We interpret this decomposition of the geodesic $\gamma_{p/q}$ into loops as showing that it {\em winds}  a total of  $p$ times around $\gamma_{0}$ and   a total of $q$ times around $\gamma_{\infty}$.  Note however, that the order in which it goes around each of  the loops $\gamma_{0}$ and $\gamma_{\infty}$ is determined by the form of the primitive word $\EE_{p/q}$.    The integer entries in the continued fraction expansion for $p/q$  determine the number of consecutive times it winds about each curve before it winds around the other curve. There are words in $\GG$ in which $\AA$ appears $p$ times and $\BB$ appears $q$ times that are not primitive.   Our discussion does not apply to them.
 \end{remark}
\begin{remark}  Note that we obtain this geometric interpretation only because we started with stopping generators.  \end{remark}

\section{$\HH^3$: Generalized ESI's}

\subsection{Transversals in the quotient }

  The ESI points for our model group $\GG$ have a  geometric interpretation on the quotient surface.   Definition~\ref{esi2}, however, makes sense for an arbitrary free discrete group $G=\langle A,B \rangle \subset \Isom(\HH^3)$ although the geometric interpretation is not clear because, generically, the pairs of lines do not intersect.

  In the quotient manifold $M=\HH^3$ we can look at the projections $\gamma_{p/q}$ of of the axes of the words $E_{p/q}$ and the projections  of the half-turn lines $L, L_A, L_B$.  Adapting the notation in \cite{GKhyp,GKwords} we call them $\ell, \ell_A, \ell_B$.
   In \cite{GKhyp} we proved that $\ell, \ell_A, \ell_B$ are three bi-infinite geodesics in $M$ and that the half-turn about any one of them projects to an involution $\INV$ of $M$ whose fixed set is precisely these three geodesics.
  We also proved that for any primitive element $E_{p/q}$   $G$,   $\gamma_{p/q}$ is invariant (but not pointwise invariant) under the involution $\INV$.

  Suppose for some $\xi$ there is a geodesic $\omega$ in $M$ joining $\xi$ to  $\INV(\xi) \neq \xi$ and suppose $\omega$ meets $\gamma_{p/q}$ orthogonally at these points.  We call such an $\omega$ a {\em transversal} since $\omega$ must also intersect one of $\ell, \ell_A, \ell_B$ orthogonally.   For the sake of discussion, assume it is $\ell$.
  Then we can lift $\gamma_{p/q}$ to $Ax_{E_{p/q}}$ and find a lift of $\ell$, $W(L)$ for some $W \in G$ such that $\omega$ lifts to the common orthogonal $O$ to  $Ax_{E_{p/q}}$ and $W(L)$.

  Clearly $W$  is not uniquely determined and $O$ depends on the choice of $W$.   If, however, $W(L)=L^{p/q}_j$ for some $j$, then $O$ can be interpreted as a generalization of the intersection point.   In general, the number of transversals  we can find for a given $\gamma_{p/q}$ is unrelated to $p/q$.  Since our goal is to relate the generalized intersection points to $p/q$, at least for certain groups, we proceed below to define our generalization from the labelling of the model group transferred to the abstract group via the natural isomorphism rather than the geometry of $M$.

 \subsection{  Generalized ESI's for a group $G \in \HH^3$}

 We now consider the group $G=\langle A,B \rangle$  as an arbitrary free purely loxodromic discrete group in $\Isom(\HH^3)$ and let and its degree two extension be  $\tilde{G}=\langle H_L, H_{L_A},H_{L_B} \rangle$.  As such, we can talk about axes and half-turn lines.  We can transfer the labelling we have for the model group to the group $G$:  words $E_{p/q}$ in $G$ are just the words $\EE_{p/q}$ with $\AA$ replaced by $A$ and $\BB$ replaced by $B$.   Since each half-turn line $L^{p/q}_j$ for $\GG$ can be written as $\WW(\LL_{\XX_j})$ with $\LL_{\XX_{j}} \in \{ \LL, \LLA,\LLB\}$,  we label a subset of  half-turn lines for $G$ again by   $L^{p/q}_j=W(L_{X_j})$,    where  $X_j$ corresponds to $\XX_j$ under the isomorphism.

The equivalence relation between pairs also transfers to $G$ in the obvious way:  $(L^{p/q}_j,Ax_{E_{p/q}})$ is equivalent to $(L_X,Ax_{Y})$ for some half-turn line $L_X$ of $\tilde{G}$ and some primitive element of $Y \in G$ if there is a  $W \in G$ such that $W(L^{p/q}_j)=L_X$ and  $W(Ax_{(E_{p/q})}) =Ax_{E_{p/q}}$.  Again using square brackets for equivalence classes, setting $n-[p+q-1]/2$ and defining $\bigsqcup$ as above we have

\begin{defn}\label{combESI} Given $G=\langle A,B \rangle \subset \Isom(\HH^3)$.  The {\em  Generalized ESI set} for $G$ is defined as the collection of equivalence classes of pairs of geodesics:
\[ SESI(G)=\bigcup_{p/q \in \QQ} \{[ESI_{E_{p/q}} ] \} =  {\bigsqcup}_{j=-2n-1}^{2n+1}[(L_j(p/q),Ax_{E_{p/q}})] \]
where the labeling and notation is taken from the model group $\GG=\langle \AA,\BB \rangle$ with stopping generators $\AA,\BB$.
\end{defn}

 \begin{thm} The set $SESI(G)$ is well-defined for any free, purely loxodromic discrete group $G=\langle A,B \rangle \subset \Isom(\HH^3).$
 \end{thm}

 \begin{proof}  The set $SESI(G)$ is defined by pairs of geodesics.  These are defined by their endpoints in $\partial{\HH^3}$.  For the set to be well-defined, these endpoints need to be distinct.   This will be true of the axes if the group is free and all its elements are loxodromic.  The half-turn lines will have distinct endpoints if the group is discrete.
 \end{proof}

  \subsection{The  Connectors}

Let  $(L_j^{p/q},Ax_{E_{p/q}})$ be a representative pair   for an equivalence class in $SESI(G)$.   In general the two geodesics in a pair will not intersect so   we consider their common orthogonal,   $O_j^{p/q}$.    Its endpoints can be computed using cross ratios from the endpoints of the geodesics of the pair.
      Thus, if the  pair   $(L_j^{p/q},Ax_{E_{p/q}})$ does intersect,  the $O_j^{p/q}$ determined by the cross ratios is  the common orthogonal through the intersection point to the plane they span.  Thus for the model group $\GG$, this orthogonal $O_j^{p/q}$ is a geodesic perpendicular to the plane left invariant by $\GG$ and intersecting it at the point $Q^{p/q}_j$;
   it is therefore the  appropriate generalization of the intersection point.

  Continuing our use of bar notation, we set $\overline{E}^j_{p/q}=H_{L_j^{p/q}}E_{p/q}H_{L_j^{p/q}}$.  Then since $G$ has index two in $\tilde{G}$,
  \[   (L_j^{p/q},Ax_{\overline{E}^j_{p/q}}) \in [(L_j^{p/q},Ax_{E_{p/q}})].  \]
  Since  $O_j^{p/q}$ is perpendicular to $L_j^{p/q}$,   $H_{L_j^{p/q}}(O_j^{p/q})=O_j^{p/q}$.  That is $O_j^{p/q}$ is also  the common orthogonal  for the pair $(L_j^{p/q},Ax_{\overline{E}^j_{p/q}})$.   This leads us to define

  \begin{defn}\label{connector} For each representative pair of geodesics $(L_j^{p/q},Ax_{E_{p/q}})$ in $SESI(G)$, the common orthogonal $O_j^{p/q}$ to the pair is a geodesic called a {\em  connector}.   Note that the connectors depend on the presentation of $G$ and the labelling induced by the isomorphism to the model group.
  \end{defn}

  The axes of $E_{p/q}$ and $\overline{E}^j_{p/q}$ project to the same geodesic $\gamma_{p/q} $ in $M$ and the connector $O_j^{p/q}$ projects to a geodesic  $\omega_j^{p/q}$ joining two points on   $\gamma_{p/q}$ and meeting it orthogonally at both these points.   The geodesic $\omega_j^{p/q}$ also intersects the projection of the half-turn line $L_j^{p/q}$ orthogonally.  This is one of the geodesics $\ell, \ell_A, \ell_B$.  Otherwise said,  the connectors project to  transversals in the quotient $M$.

 Thus we have proved Theorem~A.  We state it in two parts.

 \begin{thm}\label{thm:connectors}  [Theorem A in $\HH^3$] Let $G=\langle A,B \rangle \subset \Isom(\HH^3)$ be discrete, free and   purely loxodromic.   For each rational $p/q$, there is a geodesic    $\gamma_{p/q}$ in the quotient manifold $\HH^3/G$ with $2(p+q-1)$ marked points joined in pairs by geodesics
$\omega_j^{p/q}$, $j=2, \ldots, p+q$ that meet $\gamma_{p/q}$ orthogonally.
\end{thm}

 We also have

 \begin{thm}\label{thm:countloops} For $G$ as above, the union of the curves $\gamma_{p/q}$ and  $\omega_j^{p/q}$, $j=1, \ldots, p+q-1$ naturally divides into a set consisting of  $p+q$ non-homotopically trivial loops.  The order in which the same loop occurs a consecutive number of times is determined by the entries on the continued fraction expansion of $p/q$.
 \end{thm}
 \begin{proof}  Denote the endpoints of  $\omega_j^{p/q}$ on $\gamma_{p/q}$  by $\xi_j, \xi_j'$ so that they are joined by the projection of the  connector $\omega_j^{p/q}$.   Then following $\gamma_{p/q}$ from $\xi_j$ to $\xi_{j+1}$, next following $\omega_{j+1}^{p/q}$ to $\xi_{j+1}'$, next following $\gamma_{p/q}$ to $\xi_j'$ and finally following $\omega_j^{p/q}$ from $\xi_j'$ to $\xi_j$ we obtain a closed loop.   Note that since each loop passes once through two distinct  geodesics in the invariant set of the involution,  $\{\ell, \ell_A, \ell_B\}$,  it cannot be homotopically trivial.

   \end{proof}

  We have used the presentation of the group to define the connectors.  Note that while each connector is a transversal, there are many transversals that are not connectors.  For example,  in the model group $\GG$, $\gamma_0$ has no connectors but many transversals.   One is the projection of the common orthgonal to $Ax_A$ and $L_B$.

  \section{  Winding Groups}

  In this section we define a  class of groups, more general that the class of the model group,   where   the loops formed by the connectors are  a generalization of {\sl winding} around the generators.   Key properties for these groups are that Propositions~\ref{thru F} and \ref{intersectL} hold.   To define these groups we need some preliminaries.

 We denote the {\em convex hull} of the limit set of $G$ in $\HH^3$ by $\mathcal C$.  It is invariant under $G$ and the quotient $\mathcal C/G$ is  the {\em convex core} of $G$.   For the model group $\GG$ with stopping generators, the convex hull of the limit set is formed by deleting all the half-planes bounded by the intervals of discontinuity of $\GG$.  These are determined by the endpoints of the axes of $\AA,\BB,\AA^{-1}\BB$ and all of their conjugates.  The convex core is the surface $S=\HH^2/\GG$ truncated along the geodesics $\gamma_0, \gamma_1$ and $\gamma_{\infty}$.

 The boundary of the convex core $\tilde{S}=\partial\mathcal C/G$ is a surface of genus two.  It is what is known as a {\em pleated surface}.  For the purposes of this paper, what this means is that if more than one geodesic passes through a point on $\tilde{S}$ infinitely many do; that is, the surface looks like a piece of hyperbolic plane at that point.   Points through which only one geodesic passes form the {\em pleating locus} of $\tilde{S}$.  This consists of infinite open geodesics and/or closed simple geodesics. These are also called pleating curves.    If there is a closed geodesic in the pleating locus, it is the projection of the axis of an element of $G$ and this axis lies on the boundary of $\mathcal C$.

 \begin{defn}  The pleating locus of $\tilde{S}$ is called rational if it consists entirely of closed geodesics.  \end{defn}

A special class of groups whose convex core boundary has rational pleating locus are what we term {\em Winding Groups.}
\begin{defn} Let $G=\langle A,B \rangle$ be a group and suppose $\tilde{S}=\partial\mathcal C/G$ has rational pleating locus where the pleating curves are projections of the axes of $A,B, A^{-1}B$. Then $G$ is called a {\em winding group}. \end{defn}

  We have
  \begin{prop}\label{winding groups}  If $G$ is a winding group, then any palindromic word in the palindromic enumeration scheme $E_{p/q}$ ($p+q$ odd) has an axis that intersects the common perpendicular $L$ to the axes of $A$ and $B$ between its intersections with $A$ and $B$.   If $p+q$ is even, the common perpendicular to  $\tilde{E}_{p/q}=H_{L}E_{p/q}H_{L}$ and $E_{p/q}$ intersects the axes of $A$ and $B$ between its intersections with $A$ and $B$.
  \end{prop}
  \begin{proof}  This follows from the definition of the convex hull.  No axis can intersect the boundary of the convex hull transversally; if it intersects the boundary it lies entirely in the boundary.  If $p+q$ is odd, the palindromic property implies the axis of $E_{p/q}$ intersects $L$ so it must do it inside the convex hull.   If $p+q$ is even, the axes of $\tilde{E}_{p/q}$ and $E_{p/q}$ lie inside the convex hull.  By the symmetry in the definition their  common orthogonal intersects $L$ and by the convexity the intersection point is inside the convex hull.
  \end{proof}

  The hexagons $\H(A,B), \overline{\H(A,B)}$ are well defined for $G$ and so we can think of  their union $\F$ as a piecewise geodesic polygon in $\HH^3$, but not a fundamental domain.   The lines $L^{p/q}_j$ are sides of translates of these hexagons.  For a winding group, the axis sides of the hexagon project to pleating curves.
  The content of Proposition~\ref{winding groups} is that Propositions~\ref{thru F} and \ref{intersectL} hold for winding groups with this definition of $\F$.   The generators of the winding groups play the role of the stopping generators of the model group.

  Thus we have a geometric interpretation of Theorems~\ref{thm1} and ~\ref{thm:countloops} for these groups.    We summarize this as
  \begin{thm}\label{thm:summary}[Theorem B] Given a winding group $G$,   for every primitive geodesic in the quotient manifold, there is an analogue of the winding we found in the two dimensional case that reflects the rational number identified with it.   \end{thm}

 \bibliographystyle{amsplain}

\end{document}